\documentclass{article}

\usepackage{amsfonts,amsmath, amsthm, amssymb}
\usepackage[latin1]{inputenc}
\usepackage[english]{babel}
\usepackage{graphicx}
\usepackage{enumerate}

\def\td{tree-decom\-po\-si\-tion}

\newcommand{\tw}{\rm tw}
\newcommand{\ctw}{\rm ctw}
\newcommand{\cbn}{\rm cbn}
\newcommand{\wctw}{\rm wctw}
\newcommand{\lbound}{(\ell(G)-2)}
\newcommand{\B}{\mathcal B}
\newcommand{\U}{\mathcal U}
\newcommand{\V}{\mathcal V}
\newcommand{\W}{\mathcal W}
\newcommand{\C}{\mathcal C}

\newcommand{\sub}{\subseteq}

\newcommand{\comment}[1]{}

\newtheorem{theorem}{Theorem}%[section]
\newtheorem{lemma}[theorem]{Lemma}
\newtheorem{corollary}[theorem]{Corollary}
\newtheorem{proposition}[theorem]{Proposition}

\theoremstyle{definition}
\newtheorem{question}{Problem}

\def\?#1{\vadjust{\vbox to 0pt{\vss\vskip-8pt\leftline{%
     \llap{\hbox{\vbox{\pretolerance=-1
     \doublehyphendemerits=0\finalhyphendemerits=0
     \hsize16truemm\tolerance=10000\small
     \lineskip=0pt\lineskiplimit=0pt
     \rightskip=0pt plus16truemm\baselineskip8pt\noindent
     \hskip0pt        %(without this, the first word is never hyphenated!)
     #1\endgraf}\hskip7truemm}}}\vss}}}

\title{Bounding connected tree-width}
\author{Matthias Hamann \and Daniel Wei{\ss}auer\smallskip \and Department of Mathematics\\ University of Hamburg}
\date{}

\begin{document}
  
 \maketitle
 
  \begin{abstract}
Diestel and M\"uller showed that the \emph{connected tree-width} of a\linebreak graph~$G$, i.\,e., the minimum width of any \td\ with connected parts, can be bounded in terms of the tree-width of~$G$ and the largest length of a geodesic cycle in~$G$.
We improve their bound to one that is of correct order of magnitude. Finally, we construct a graph whose connected tree-width exceeds the connected order of any of its brambles.
 This disproves a conjecture by Diestel and M\"uller asserting an analogue of tree-width duality.
 \end{abstract}

  \begin{section}{Introduction} \label{intro}
 Intuitively, a \td\ $(T, (V_t)_{t \in T})$ of a graph~$G$ can be regarded as giving a bird's-eye view on the global structure of the graph, represented by~$T$, while each part represents local information about the graph.  But this interpretation can be misleading: the tree-decomposition may have disconnected parts, containing vertices which lie at great distance in~$G$, and so this intuitively appealing distinction between local and global structure can not be maintained.
 
 This can be remedied if we require every part to be connected. We call such a \td\ \emph{connected}. Jegou and Terrioux \cite{JegouTerrioux2014preprint,JegouTerrioux2014} pointed out that the efficiency of algorithmic methods based on \td s for solving constraint satisfaction problems can be improved when using connected \td s. 
 
 The \emph{connected tree-width} $\ctw (G)$ is defined accordingly as the minimum width of a connected \td\ of the graph~$G$. Trivially, the connected tree-width of a graph is at least as large as its tree-width and, as Jegou and Terrioux \cite{JegouTerrioux2014} observed, long cycles are examples of graphs of small tree-width but large connected tree-width. Diestel and M\"uller \cite{MalteRhd} showed that, more generally, the existence of long \emph{geodesic} cycles, that is, cycles in a graph~$G$ that contain a shortest path in~$G$ between any two of their vertices, raises the connected tree-width. Furthmore, they proved that these two obstructions to small connected tree-width, namely, large tree-width and long geodesic cycles, are essentially the only obstructions:
 \pagebreak
  
  \begin{theorem}[{{\cite[Theorem~1.1]{MalteRhd}}}]
  There is a function $f: \mathbb{N}^2 \to \mathbb{N}$ such that the connected tree-width of any graph of tree-width~$k$ and without geodesic cycles of length greater than~$\ell$ is at most $f(k, \ell)$.
  \end{theorem}
  
  They also showed that $f(k, \ell ) = {\cal O}(k^3 \ell)$. In fact, their proof does not only work with geodesic cycles, but with any collection of cycles that generate the cycle space of the graph~$G$. Given a graph~$G$, we define~$\ell (G)$ to be the smallest natural number~$\ell$ such that the cycles of length at most~$\ell$ generate the cycle space of~$G$.
Our main result improves the bound of Diestel and M\"uller significantly:
  
  \begin{theorem} \label{main result}
   Let $G$ be a graph containing a cycle. Then the connected tree-width of~$G$ is at most $\tw(G)\lbound$.
  \end{theorem}

(Observe that every forest satisfies $\ctw (G) = \tw(G) \leq 1$.)
Theorem~\ref{main result} will be proved in Sections~\ref{defs}--\ref{bound}. In Section~\ref{example} we discuss an example that demonstrates that this bound is best possible up to a constant factor.

Note that $\ell(G)$ can differ arbitrarily from the length of a longest geodesic cycle: consider e.\,g.\ an $(n \times n)$-grid where every edge except for those on the boundary is subdivided once. Then the boundary is a geodesic cycle of length~$4(n-1)$, while the cycle space is generated by the collection of `squares', each of length at most~$8$.
%At the same time, a large gap between $\ell(G)$ and the length of a longest geodesic cycle can serve as a certificate for large tree-width:
It is no coincidence that the graph in this example has large tree-width, as the following unexpected consequence of our inquiry shows:

\begin{corollary}\label{cycle space tw}
Every graph~$G$ containing a geodesic cycle of length~$k$ has tree-width at least $k/\ell(G)$. 
\end{corollary}

The \emph{tree-width duality theorem} of Seymour and Thomas~\cite{ST1993GraphSearching} asserts that a graph has tree-width less than~$k$ if and only if it has no bramble of order at least~$k$. Diestel and M\"uller~\cite{MalteRhd} conjectured that a similar duality holds for connected tree-width and the maximum \emph{connected order} of a bramble: the minimum size of a connected vertex set meeting every element of the bramble. 
%Diestel and M\"uller \cite[Theorem~1.2]{MalteRhd} showed that the connected tree-width of a graph is bounded by a function of the maximum connected order of its brambles. 
We disprove their conjecture by giving an infinite family of counterexamples in Section~\ref{duality}.

 Since every \td\ has, for every bramble of the graph, a part covering it, Theorem~\ref{main result} immediately yields an upper bound on the connected order of any bramble of the graph. In Section~\ref{brambles}, we apply the techniques and results from previous sections to strengthen this bound:

\begin{theorem} \label{bramble bound}
	Let $G$ be a graph containing a cycle. Then the connected order of any bramble of $G$ is at most $\tw(G)\lfloor \frac{\ell(G)}{2} \rfloor + 1$.
\end{theorem}

 \end{section}

 \begin{section}{Definitions and notation} \label{defs}
  For a tree~$T$ with root~$r$, we call~$s$ a \emph{descendant} of~$t$ and~$t$ an \emph{ancestor} of~$s$ if~$t$ lies on the unique path from~$r$ to~$s$. 
  If additionally $st \in E(T)$, we call~$s$ a \emph{child} of~$t$ and~$t$ the \emph{parent} of~$s$. 
  We write $T_t$ for the subtree of descendants of~$t$.
  Recall that a \emph{\td} of $G$ is a pair $(T,\V)$ of a tree $T$ and a family $\V=(V_t)_{t\in T}$ of vertex sets $V_t\sub V(G)$, one for every node of~$T$, such that:
\begin{enumerate}[(T1)]
\item $V(G) = \bigcup_{t\in T}V_t$,
\item for every edge $e$ of~$G$ there exists a $t \in T$ with $e\sub V_t$,
\item $V_{t_1} \cap V_{t_3} \sub V_{t_2}$ whenever $t_2$ lies on the $t_1$--$t_3$ path in~$T$.
\end{enumerate}

The sets $V_t$ in such a \td\ are its \emph{parts}.
 For $A \sub V(T)$ we write $V_A := \bigcup_{t \in A} V_t$. 
The \emph{width} of $(T,\V)$ is $\max_{t\in T}(|V_t|-1)$ and the \emph{tree-width} $\tw(G)$ of~$G$ is the minimum width of any of its \td s.

  In our proof of Theorem~\ref{main result} we will make use of an explicit procedure that transforms a given \td\ into a connected \td\ by iteratively adding paths to a disconnected part of the decomposition. For this to work efficiently, we will restrict ourselves to paths of a particular kind.
  
   Let $(T, \V)$ be a \emph{rooted} \td\ of~$G$, i.\,e.\ $T$ is rooted, and $t \in T$. A path $P$ in $G$ is \emph{$t$-admissible} if it lies entirely in $V_{T_t}$, joins different components of $V_t$ and is shortest possible with these properties.
  Note that $t$-admissible paths have precisely two vertices in~$V_t$:
  
  \begin{lemma} \label{path shape}
   Let $(T, \V)$ be a rooted \td\ of a graph $G$, $t \in T$ and $P$ a $t$-admissible path. Then there is a unique child $s$ of $t$ such that all internal vertices of $P$ lie in $V_{T_s} \setminus V_t$.\qed
  \end{lemma}

  In general, $t$-admissible paths need not exist. However, as we shall see, we can easily confine ourselves to \td s that always have $t$-admissible paths.
  
   We call a \td\ $(T, \V)$ \emph{stable} if for every edge $t_1t_2 \in E(T)$ of $T$, both $V_{T_1}$ and $V_{T_2}$ are connected in $G$, where $T_i$, for $i = 1,2$, is the component of $T - t_1t_2$ containing $t_i$. (Later, we will use this naming convention without further mention.)

\begin{lemma}\label{stable give admissible}
Let $(T,\V)$ be a rooted stable \td\ of a connected graph~$G$.
Then every $t\in T$ with disconnected $V_t$ has a $t$-admissible path.\qed
\end{lemma}
   
   Stable \td s were also studied in~\cite{FraigniaudNisse}, where they are called \emph{connected \td s}.
In that article, an explicit algorithm is presented that turns a \td\ of a connected graph into a stable \td\ without increasing its width.
For our purposes it suffices to know that every connected graph has a stable \td\ of minimum width.
This can also be deduced from~\cite[Corollary~3.5]{MalteRhd}.

\begin{proposition}\label{stable optimal exists}
Every connected graph $G$ has a stable \td\ of width $\tw(G)$.\qed
\end{proposition}
  
   If we add a $t$-admissible path $P$ to a part $V_t$ in order to join two of its components, we might not obtain a \td. The following lemma shows how it can be patched.   
  \begin{lemma} \label{add path}
   Let $(T, \V)$ be a rooted \td\ of a graph $G$, $t \in T$ and $P$ a $t$-admissible path. For $u \in T$ let  
   \begin{equation} \label{updating}
    W_u := \begin{cases}
            V_u \cup (V(P) \cap V_{T_u} ), &\text{ if } u \in T_t, \\
            V_u, &\text{ if } u \notin T_t. \\
           \end{cases}\tag{$*$}
   \end{equation}
  Then $(T, \W)$ is a \td\ of $G$. For all $u \in T$, every component of $W_u$ contains a vertex of $V_u$. If $(T, \V)$ is stable, so is $(T, \W)$.
  \end{lemma}
\begin{proof}
 Since $V_u \sub W_u$ for all $u \in T$, every vertex and every edge of $G$ is contained in some part $W_u$. 
 
 Let $I$ be the set of internal vertices of $P$. By Lemma~\ref{path shape} there is a unique child $s$ of $t$ such that $I \sub V_{T_s} \setminus V_t$. For $x \notin I$, the set of parts containing $x$ has not changed. For $x \in I$, the set $A_x := \{ u \in T\colon x \in V_u \}$ induces a subtree of $T_s$ and $x \in W_u$ if and only if $u \in A_x$ or $u$ lies on the path joining $t$ to $A_x$. So $\{ u\colon x \in W_u \}$ is also a subtree of $T$.
 
 Note that every component of $P \cap V_{T_u}$ is a path with ends in $V_u$. Therefore every $x \in W_u \setminus V_u$ is joined to two vertices in $V_u$ and thus every component of $W_u$ contains vertices from $V_u$.
 
 Suppose now $(T, \V)$ is stable, let $t_1t_2 \in E(T)$ and $i \in \{ 1, 2\}$. Then $V_{T_i}$ is connected. For $x \in W_{T_i} \setminus V_{T_i}$ there is a $u \in T_i \cap T_t$ with $x \in W_u \setminus V_u$. But then, by the above, $W_u$ contains a path joining $x$ to $V_{T_i}$. As $V_{T_i} \sub W_{T_i}$, also $W_{T_i}$ is connected.
\end{proof}

  \end{section}

  \begin{section}{The construction} \label{algo}

We now describe a construction that turns a stable \td\ $(T, \V)$ of a connected graph into a connected \td.
First, choose a root $r$ for $T$ and keep it fixed.
It will be crucial to our analysis that the nodes of~$T$ are processed in the induced order of the tree, i.\,e.\ we enumerate the nodes $t_1,t_2,\ldots$ so that each node precedes its descendants and we process the nodes in this order.

Initially we set $W_t = V_t$ for all $t \in T$. Throughout the construction, we maintain the invariant that $(T, \W)$ is a stable \td\ \emph{extending} $(T, \V)$, by which we mean that they are \td s over the same rooted tree, satisfying $V_t \sub W_t$ for all $t \in T$.

 When processing a node $t \in T$ with disconnected part $W_t$, we use the stability of $(T, \W)$ to find a $t$-admissible path by Lemma~\ref{stable give admissible} and update $\W$ as in~(\ref{updating}).
By Lemma~\ref{add path}, this does not violate stability and it clearly reduces the number of components of $W_t$ by one.
We iterate this until $W_t$ is connected. Once that is achieved, we continue with the next node in our enumeration. 
 
Observe that each `update' only affects descendants of the current node.
Once a node $t \in T$ has been processed, so have all of its ancestors.
Hence, no further changes are made to~$W_t$ afterwards.
In particular, $W_t$ remains connected. It thus follows that, when every node has been processed, the resulting \td\ is indeed connected.
   
In order to control the size of each part $W_u$, we will use a \emph{bookkeeping graph} $Q_u$ to keep track of what we have added.
Initially, $Q_u$ is the empty graph on~$V_u$, and in each step $Q_u$ is a graph on the vertices of~$W_u$.
Whenever something is added to~$W_u$, we are considering a $t$-admissible path~$P$ for some ancestor $t$ of~$u$ and $P$ contains vertices of~$W_{T_u}$.
Every component of $P \cap W_{T_u}$ is a path with ends (and possibly also some internal vertices) in~$W_u$.
We then add $P \cap W_{T_u}$ to~$Q_u$, that is, we add all the vertices not contained in $W_u$ and all the edges of $P \cap W_{T_u}$.
  
\begin{comment}{
  \begin{figure}[h]
 \begin{algorithm}[H]
 \KwData{a stable \td\ $(T,\V)$ }
 \KwResult{a connected \td\ $(T,\W)$ }
 initialize $W_t := V_t$, $Q_t$ the empty graph on $V_t$\;
 order $V(T) = \{ t_1, \ldots, t_n \}$ such that if $t_i$ is an ancestor of $t_j$, then $i < j$ \;
 \For{$t = t_1, \ldots , t_n$}{
  \While{$W_{t}$ is not connected}{
    find a $t$-admissible path $P$ \;
    update $\W$ as in Lemma \ref{add path} \;
    \ForEach{$s \in T_{t}$}{
      add $P \cap W_{T_s}$ to $Q_s$ \;
    }
  }
  }
  \caption{Greedy algorithm for connected \td s}
\end{algorithm}
 \end{figure}
}\end{comment}
  
  \begin{lemma} \label{Q_t acyclic}
   During every step of the procedure, $Q_u$ is acyclic.
  \end{lemma}
  
  \begin{proof}
   This is certainly true initially. Suppose now that at some step a cycle is formed in $Q_u$. By definition, it must be that an ancestor $t$ of $u$ is being processed and a $t$-admissible path $P$ is added such that two vertices $a, b \in W_u$ which were already connected in $Q_u$ lie in the same component of $P \cap W_{T_u}$.
   
   The vertices $a, b$ being connected in $Q_u$ by a path $a = a_0a_1\ldots a_n = b$ means that there have been, for every $0 \leq j \leq n-1$, ancestors $t_j$ of $u$ that added paths $P_j$ such that $a_j$, $a_{j+1}$ were consecutive vertices on a segment $S_j$ of $P_j \cap W_{T_u}$. By the order in which the nodes are processed and by~(\ref{updating}), these $t_j$ are also ancestors of $t$. Therefore when $P_j$ was added to $W_{t_j}$, the segment $S_j$ was contained in a segment of $P_j \cap W_{T_{t}}$, since $W_{T_{t}} \supseteq W_{T_u}$. Therefore, at the time $P$ is added to $W_{t}$, all these segments are contained in $W_{t}$ and, in particular, $a, b \in W_{t}$. By Lemma~\ref{path shape}, $P$ does not have internal vertices in $W_{t}$ so that $a$ and $b$ must in fact be the ends of $P$. But $W_{t}$ already contains a walk from $a$ to $b$, consisting of the segments $S_j$, so that the two do not lie in different components of $W_{t}$, contradicting the $t$-admissibility of~$P$. 
  \end{proof}

  We now show how the sparse structure of $Q_u$ reflects the efficiency of our procedure.
  
  \begin{lemma} \label{bookkeeping}
   The number of components of $Q_u$ never increases. Whenever something is added to $W_u$, the number of components of $Q_u$ decreases.
  \end{lemma}
  
  \begin{proof}
   Suppose that in an iteration a change is made to $Q_u$. Then an ancestor $t$ of $u$ is being processed and the chosen path $P$ meets $W_{T_u}$. Every component of $P \cap W_{T_u}$ is a path with both ends in~$W_u$. Therefore, every newly introduced vertex is joined to a vertex in $Q_u$ and no new components are created.
   
   If a vertex from $P \cap (W_{T_u} \setminus W_u)$ is added to $W_u$, the segment containing it has length at least two and has two ends $a, b \in Q_u$. By Lemma~\ref{Q_t acyclic}, $Q_u$ must remain acyclic, so that $a$ and $b$ in fact lie in different components of~$Q_u$, which are now joined. 
  \end{proof}

The previous lemma allows us to control the number of iterations that affect a fixed node $t\in T$.
The second key ingredient for the proof of Theorem~\ref{main result} will be to bound the length of each of the paths used, see Section~\ref{bound}.

  \begin{proposition} \label{metatheorem}
  
Let $G$ be a connected graph, $(T, \V)$ a rooted stable \td\ of~$G$. For $t \in T$ let $m_t \geq 1$ be such that for every stable \td\ $(T, \W)$ extending $(T, \V)$ and every ancestor~$t'$ of~$t$, the length of a $t'$-admissible path in $(T, \W)$ does not exceed~$m_t$. Then the construction produces a connected \td\ $(T, \U)$ in which for all $t \in T$
\begin{equation*}
	|U_t| \leq m_t(|V_t| - 1) + 1.
\end{equation*}
\end{proposition}

  \begin{proof}
   We have already shown that $(T,\U)$ is connected. By Lemma~\ref{bookkeeping}, every time something was added to $W_t$, the number of components of $Q_t$ decreased and it never increased. Since initially $Q_t$ had precisely $|V_t|$ components, this can only have happened at most $|V_t| - 1$ times. In each such iteration we added some internal vertices of a $t'$-admissible path in a stable \td\ extending $(T, \V)$ for some ancestor~$t'$ of~$t$, thus at most $m_t-1$ vertices. In total, we have
   \begin{equation*}
    |U_t| \leq |V_t| + (m_t-1)(|V_t| - 1) = m_t(|V_t| - 1) + 1.\qedhere
   \end{equation*}
  \end{proof}

%\begin{theorem} \label{metatheorem}
%
% Let $G$ be a graph and let $m$ be such that for every stable \td\ $(T, \W)$ of $G$ and every $t \in T$, the length of every $t$-admissible path $P$ is at most~$m$. Given a stable \td\ $(T, \V)$, the construction produces a connected \td\ $(T, \U)$ in which for all $t \in T$
%   \begin{equation*}
%    |U_t| \leq m(|V_t| - 1) + 1.
%   \end{equation*}
%  \end{theorem}
%  
%  \begin{proof}
%   We have already discussed the correctness of the construction\?{\tiny Diesen Satz nach \"Uberarbeitung von Algo erneut ansehen}. By Lemma~\ref{bookkeeping}, every time something was added to $W_t$, the number of components of $Q_t$ decreased and it never increased. Since initially $Q_t$ had precisely $|V_t|$ components, this can only have happened at most $|V_t| - 1$ times. In each such iteration we added some internal vertices of a $t'$-admissible path in a stable \td, thus at most $m-1$ vertices. In total, we have
%   \begin{equation*}
%    |U_t| \leq |V_t| + (m-1)(|V_t| - 1) = m(|V_t| - 1) + 1.\qedhere
%   \end{equation*}
%  \end{proof}

\end{section}

\begin{section}{Bounding the length of admissible paths}\label{bound}
 We will now use ideas from~\cite{MalteRhd} to bound the length of $t$-admissible paths in stable \td s. Together with Proposition~\ref{metatheorem}, this will imply our main result.
 
 \begin{lemma} \label{geo cycle}
  Let~$G$ be a graph and~$\Gamma$ a set of cycles that generates its cycle space. Let $(T, \V)$ be a stable \td\ of ~$G$ and $t_1t_2 \in E(T)$. Suppose that $V_{t_1}\cap V_{t_2}$ meets two distinct components of~$V_{t_1}$. Then there is a cycle~$C \in \Gamma$ such that some component of $C \cap V_{T_2}$ meets~$V_{t_1}$ in two distinct components.
 \end{lemma}
 
  \begin{proof}
   As~$V_{T_2}$ is connected, we can choose a shortest path $P$ in $V_{T_2}$ joining two components of~$V_{t_1}$. Let $x, y \in V_{t_1}$ be its ends and note that all internal vertices of~$P$ lie in $V_{T_2} \setminus V_{t_1}$. As~$V_{T_1}$ is connected as well, we also find a path $Q \sub V_{T_1}$ joining~$x$ and~$y$, which is internally disjoint from~$P$. By assumption, there is a subset~$\C$ of~$\Gamma$ such that $P + Q = \bigoplus \C$. We subdivide~$\C$ as follows: $\C_1$ comprises all those cycles which are entirely contained in $V_{T_1} \setminus V_{T_2}$, $\C_2$~those in $V_{T_2} \setminus V_{T_1}$ and~$\C_X$ those that meet $X:=V_{t_1}\cap V_{t_2}$.
   
   Assume now for a contradiction that for every $C \in \C_X$ and every component~$S$ of $C \cap V_{T_2}$ there is a unique component~$D_S$ of~$V_{t_1}$ met by~$S$. Note that~$S$ is a cycle if $C \sub V_{T_2}$ and a path with ends in~$X$ otherwise. Either way, the number of edges of $S$ between $X$ and $V_{T_2} \setminus X$, denoted by $|E_S(X,V_{T_2}\setminus X)|$, is always even. It thus follows that for any component $D$ of $V_{t_1}$ %, the number of edges of $C$ between $D$ and $V_{T_2} \setminus X$
   \begin{equation*}
    |E_C(D, V_{T_2} \setminus X)| = \sum_{S \sub C \cap V_{T_2} } |E_S(D, V_{T_2} \setminus X)| = \sum_{S \colon D_S = D} |E_S(X, V_{T_2} \setminus X)|
   \end{equation*}
is even. But then also the number of edges in~$\bigoplus \C_X$ between~$D$ and $V_{T_2} \setminus X$ is even. Since the edges of~$\bigoplus \C_1$ and~$\bigoplus \C_2$ do not contain vertices from~$X$, we have
\begin{equation*}
 E_{\bigoplus \C_X}(X, V_{T_2} \setminus X) = E_{P + Q}(X, V_{T_2} \setminus X) = \{ xx', yy' \},
\end{equation*}
where~$x'$ and~$y'$ are the neighbours of~$x$ and~$y$ on~$P$, respectively. Due to parity,~$x$ and~$y$ need to lie in the same component of~$V_{t_1}$, contrary to definition.
  \end{proof}

  \begin{proof}[Proof of Theorem~\ref{main result}]
Since both parameters appearing in the bound do not increase when passing to a component of~$G$ and as we can combine connected tree-decompo\-sitions of the components to obtain a connected \td\ of~$G$, it suffices to consider the case that $G$ is connected.

   We use Lemma~\ref{geo cycle} to bound the length of $t$-admissible paths in any stable \td\ of~$G$. Let $\ell = \ell(G)$ and~$\Gamma$ be the set of all cycles of length at most~$\ell$, which by definition generates the cycle space of~$G$.
   Let $(T, \W)$ be a rooted stable \td, $t \in T$ and $P$ a $t$-admissible path. By Lemma~\ref{path shape} there is a child $s$ of $t$ such that all internal vertices of $P$ lie in $V_{T_s} \setminus V_t$. By Lemma~\ref{geo cycle} we find a cycle $C \in \Gamma$ and a path $S \sub C \cap V_{T_s}$ joining distinct components of~$V_{t_1}$. Since $S \sub V_{T_t}$ and $P$ was chosen to be a shortest such path, we have $|P| \leq |S|$. The ends of $S$ lie in distinct components of $V_{t_1}$ and are therefore, in particular, not adjacent, so that overall   
   \begin{equation*}
    |V(P)| \leq |V(S)| \leq |V(C)| - 1 \leq \ell-1.
   \end{equation*}

By Proposition~\ref{stable optimal exists}, $G$ has a stable \td\ $(T, \V)$ of width $\tw(G)$.
Proposition~\ref{metatheorem} then guarantees that we find a connected \td\ of width at most $(\ell-2)\tw(G)$.
  \end{proof}

\end{section}

\begin{section}{Brambles} \label{brambles}
%	The classical \emph{tree-width duality theorem} by Seymour and Thomas asserts that a graph has a \td\ of width less than $k$ if and only if it has no bramble of order at least $k$ (see~\cite{ST1993GraphSearching}). A \emph{bramble} is a collection of connected subgraphs such that the union of any two of them is again connected, and its \emph{order} is the minimum size of a \emph{cover}, a set of vertices that meets each element of the bramble.
%	
%	In~\cite{MalteRhd}, Diestel and M\"uller conjectured that a similar duality holds for connected tree-width\?{\tiny w\"urde ich vllt sogar in die Einleitung packen} and showed \cite[Theorem~1.2]{MalteRhd} that the connected tree-width of a graph can be bounded by a function of the maximum connected order of its brambles.
%Here, the \emph{connected order} of a bramble is the minimum size of a connected cover.
Recall that a \emph{bramble} is a collection of connected vertex sets of a given graph such that the union of any two of them is again connected. A \emph{cover} of a bramble is a set of vertices that meets every element of the bramble.
The aim of this section is to derive a strengthened upper bound on the \emph{connected order} of a bramble, the minimum size of a connected cover.

\begin{lemma} \label{easy bramble bound}
	Suppose $(T, \V)$ is a \td\ of a graph $G$ and $k \in \mathbb{N}$ an integer such that for every $t \in T$ there is a connected set of size at most~$k+1$ containing $V_t$. Then $G$ has no bramble of connected order greater than $k+1$.
\end{lemma}	
\begin{proof}
	Let $\B$ be a bramble of $G$. By a standard argument, see e.\,g.\ the proof of~\cite[Theorem 12.3.9]{DiestelBook10noEE}, one of the parts $V_t$ of $(T, \V)$ covers $\B$ and thus so does any connected set containing~$V_t$.
\end{proof}

Let us call the smallest integer~$k$ such that there is a \td\ satisfying the hypothesis of Lemma~\ref{easy bramble bound} the \emph{weak connected tree-width} $\wctw(G)$ of the graph~$G$. Clearly $\wctw(G) \leq \ctw(G)$, as any connected \td\ of minimum width satisfies the hypothesis. Theorem~\ref{bramble bound} follows directly from Lemma~\ref{easy bramble bound} and the following.

\begin{theorem} \label{wctw bound}
	Let~$G$ be a graph containing a cycle. Then 
	\[
	\wctw(G) \leq \tw(G)\lfloor \frac{\ell(G)}{2} \rfloor .
	\]
\end{theorem}

\begin{proof}
It  suffices to consider the case where~$G$ is connected, since all three parameters involved are simply their respective maxima over the components of~$G$.
Let $\ell = \ell(G)$ and~$\Gamma$ be the set of all cycles of~$G$ of length at most~$\ell$, which by definition generates the cycle space of~$G$.
 By Proposition~\ref{stable optimal exists}, $G$ has a stable \td\ $(T, V)$ of width $\tw(G)$. We now show that every part $V_t$ of $(T, \V)$ is contained in a connected set of size at most $(|V_t| - 1)\lfloor \frac{\ell}{2} \rfloor + 1$.
 
 Let now $t \in T$ be fixed. Root $T$ at $t$ and apply the construction from Section~\ref{algo}. As $t$ does not have any ancestors other than itself, the statement follows from Proposition~\ref{metatheorem} once we have verified that all $t$-admissible paths in a stable \td\ $(T, \W)$ extending $(T, \V)$ have length at most~$\ell /2$. So let $(T, \W)$ be a stable \td\ of~$G$ extending $(T, \V)$ and let~$P$ be a $t$-admissible path. By Lemma~\ref{path shape}, all its internal vertices lie in $W_{T_s} \setminus W_t$ for some child~$s$ of~$t$. By Lemma~\ref{geo cycle} we find a cycle~$C \in \Gamma$ that meets~$W_t$ in two vertices~$x, y$ from distinct components of~$W_t$. Either segment of~$C$ between~$x$ and~$y$ lies in~$W_{T_t}$ and joins two components of~$W_{t}$, so by minimality~$P$ has length at most $\lfloor \ell /2 \rfloor$.
\end{proof}

Diestel and M\"uller~\cite{MalteRhd} showed that if a graph~$G$ contains a geodesic cycle of length~$k$, then~$G$ has a bramble of connected order $\lceil k/2 \rceil + 1$. Combined with the upper bound of Theorem~\ref{bramble bound}, this implies Corollary~\ref{cycle space tw}.
	
\end{section}

\begin{section}{A graph of large connected tree-width}\label{example}
In this section we discuss an example that shows that our upper bound on connected tree-width is tight up to a constant factor. Given $n, k \in \mathbb{N}$, $n \geq 3$, obtain~$G$ from the complete graph on~$n$ vertices by subdividing every edge with~$k$ newly introduced vertices. As subdivision does not alter tree-width, we have $\tw(G) = n-1$. The cycle space of~$G$ is generated by the collection of all subdivisions of triangles of the underlying complete graph, so $\ell (G) = 3(k+1)$. We will now show that the connected tree-width of~$G$ is precisely $r := {(n-1)}{(k+1)} - \lfloor (k+1)/2 \rfloor$. The bound of Theorem~\ref{main result} is therefore asymptotically tight up to a factor of~3.

Let $A \sub V(G)$ denote the set of vertices of degree~$n-1$. The graph~$G$ thus consists of~$A$ and, for any two $a, b \in A$, a path~$P_{ab}$ of length~$k+1$ between them. We first describe a bramble that cannot be covered with any connected set of size at most~$r$. The lower bound on the connected tree-width of~$G$ then follows from Lemma~\ref{easy bramble bound}. Any connected set~$X$ consists of some vertices from~$A$, its \emph{branchvertices}, some \emph{internal vertices} $X^0$ on paths joining two of its branchvertices and possibly some \emph{additional vertices}. Any connected set with~$j$ branchvertices must have at least~$(j-1)k$ internal vertices, resulting in a minimum size of $(j-1)(k+1) + 1$. Let~$X$ be a connected set of size at most~$r$. Then, by the above, $X$ cannot contain all the vertices of~$A$ and, moreover, all vertices of $A \setminus X$ lie in the same component $C(X)$ of $G - X$: If $a, b \in A \setminus X$, then by connectedness either $X \cap P_{ab} = \emptyset$ or $X \sub P_{ab}$, in which case~$a$ and~$b$ can be joined through some other $c \in A$. Let~$\B$ be the collection of all these components $C(X)$ for $X \sub V$ connected of size at most~$r$.

Clearly, $\B$ can not be covered by any connected set of size at most~$r$, so it only remains to verify that~$\B$ is indeed a bramble. Let $X_1, X_2 \sub V$ be two connected sets of size at most~$r$, containing~$j_1, j_2$ vertices of~$A$, respectively. Suppose that $C(X_1)$ and $C(X_2)$ did not touch. Then for every pair $(a, b)$ with $a \in A \setminus X_1$, $b \in A \setminus X_2$, the sets~$X_1$ and~$X_2$ must have a common vertex on~$P_{ab}$. By definition, all these are additional vertices for both sets, so

\begin{align*}
	|X_1| + |X_2| &\geq |X_1^0| + |X_2^0| + (n-j_1)(n -j_2)(k+1) \\
	&= (k+1)( n^2 - (n-1)(j_1 + j_2) + j_1j_2 - 2) + 2 .
\end{align*}
This expression, seen as a function of~$j_1$ and~$j_2$, assumes its minimum for~$j_1 = j_2 = n-1$. We thus conclude $|X_1| + |X_2| \geq 2(k+1)(n-1) - k + 1$, hence the larger of the two sets has size at least $r+1$, a contradiction.

We now describe a connected \td\ of width~$r$. Fix two $a, b \in A$ and let $A^- := A \setminus \{ a, b \}$. Let~$T$ be a star with root~$s$ and leaves $t, u_1, \ldots , u_m$ with $m = \binom{n-2}{2}$. Each~$V_{u_i}$ consists of a different path~$P_{cd}$ with $c, d \in A^-$. Let~$V_s$ consist of the union of all~$P_{bc}$ with $c \in A^-$ and the first~$\lceil (k+1)/2 \rceil$ vertices from $P_{ba}$. Define~$V_t$ similarly.
This \td\ has the desired width. 

\end{section}

\begin{section}{A counterexample for duality}\label{duality}
In this section, we present a graph whose connected tree-width is larger than the largest connected order of any of its brambles.
Hence, we disprove the duality conjecture of Diestel and M\"uller~\cite{MalteRhd} for connected tree-width.

Let $n\geq 4$ be an integer.
For $i=0,1,2$, let $P_i=x^i_1\ldots x^i_{2n}$ be three pairwise disjoint paths and $Q=y_1\ldots y_{4n}$ another path disjoint from each~$P_i$.
Between every two vertices $x^i_j, y_k$ we add a new internally disjoint path~$P_{j,k}^i$ of length $5n$, except for $k=n+j$, where they have length~$n$.
Let $G'$ be the resulting graph.
Let $G$ be the disjoint union of $G'$ with a cycle~$C$ of length $16n+2$, where we choose two antipodal vertices $a,b$, i.\,e.\ vertices of~$C$ with $d_C(a,b)=8n+1$, and add the edges $ax^0_1, ay_1$ and $bx^0_{2n}, by_{4n}$.
Figure~\ref{fig:duality} shows the graph~$G$ without $P_2$ and its attachment paths to~$Q$.

\begin{figure}[h]
\begin{center}
\includegraphics{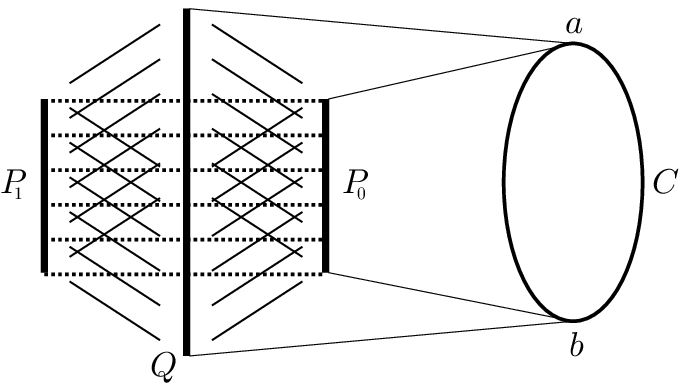}
\caption{The graph $G$ without $P_2$ and its attachment paths.}\label{fig:duality}
\end{center}
\end{figure}

We claim that the connected order of any of its brambles is at most $5n+3$ and that its connected tree-width is at least~$6n$. Thus, up to additive constants, these parameters differ at least by a factor of $6/5$.

We will now give a \td\ demonstrating that $\wctw(G) \leq 5n+2$, which is sufficient to prove the upper bound on the connected order of any bramble by Lemma~\ref{easy bramble bound}. 
Start with~$V_{t_0} := V(Q) \cup \{ a, b \}$, which is connected and of size $4n + 2$. Clearly, $G - V_{t_0}$ consists of five components: each of the~$P_i$ along with their attachments to~$V_{t_0}$ and the two arcs of~$C$. Accordingly, we add five branches to~$t_0$, each decomposing one of the components, as follows.

For $i = 0, 1,2$, attach a path $t^i_1\ldots t^i_{2n-1}$ to~$t_0$ and put $V_{t_j^i} = V_{t_0} \cup \{ x^i_j, x^i_{j+1} \}$. Each of these is contained in a connected set of size~$5n +3$, as~$x_j^i$ is joined to~$Q$ by a path of length~$n$. To each~$t^i_j$ attach~$4n$ leaves, each consisting of some~$P_{j,k}^i$, $k \in [4n]$, which obviously does not exceed the prescribed size. To~$t^i_{2n-1}$ we add another~$4n$ leaves consisting of all the~$P_{j+1,k}^i$. To decompose~$C$, we attach two more paths~$s^1_0s^1_1$ and~$s^2_0s^2_1$ to~$t_0$, one for each arc~$S^j$ of $C - \{ a, b \}$. For $j= 1,2$, $V_{s^j_0}$ contains $\{ a, b\}$ and the~$3n$ vertices of~$S^j$ which lie closest to~$a$, while~$V_{s^j_1}$ contains~$b$ and its closest $5n+1$ vertices on~$S^j$. Both of these sets are contained in connected sets of size~$5n + 2$.
Figure~\ref{fig:baum} shows our decomposition tree.

\begin{figure}[h]
\begin{center}
\includegraphics{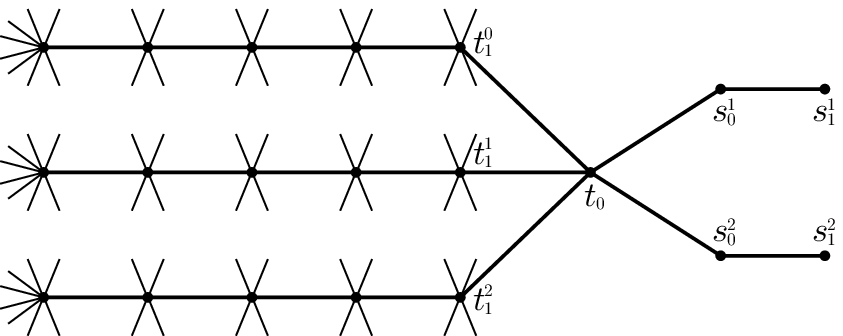}
\caption{The decomposition tree of~$G$.}\label{fig:baum}
\end{center}
\end{figure}

To show~$\ctw(G) \geq 6n$, let us assume for a contradiction that~$G$ had a connected \td\ $(T,\V)$ of width less than~$6n$.
We shall show that some part $V_t$ contains~$Q$ and some other part $V_{t'}$ contains~$P_0$.
To see that some part contains~$Q$, we define a bramble as follows: For For all $i, j, k$, let~$B^i_{j,k}$ be the union of all the paths from~$x^i_j$ to~$Q$ with all end vertices except for~$y_k$ deleted. It is easy so see that the collection~$\B_1$ of all these sets~$B^i_{j,k}$ is a bramble. Therefore, some part~$V_t$ of $(T, \V)$ must cover~$\B_1$. If some vertex~$y_k \in Q$ is not included in~$V_t$, then $V_t$ must contain at least one vertex from each of the~$6n$ pairwise disjoint sets $B^i_{j,k} \setminus \{ y_k \}$ with $i \in \{ 0, 1, 2 \}$ and $j \in [2n]$. Since no such selection of vertices is connected without the addition of further vertices, this contradicts our assumption that $|V_t| \leq 6n$.

We now show that some part contains~$P_0$. Let~$C'$ be the cycle of length $12n+2$ consisting of one of the $a$--$b$ paths on~$C$ together with~$Q$ and let~$\B_2$ be the bramble consisting of all segments of~$C'$ of length~$6n+1$. Again, there must be a part~$V_{t'}$ covering~$\B_2$. Assume for a contradiction that some vertex~$x^0_j \in P_0$ was not contained in~$V_{t'}$. Observe, crucially, that~$C'$ is geodesic in $G' := G - x^0_j$ and hence~$\B_2$ has connected order $6n+2$ in~$G'$ (see \cite[Lemma~7.1]{MalteRhd}). As~$V_{t'}$ is a cover of~$\B_2$ in~$G'$, it follows that~$V_{t'} \geq 6n+2$, which is a contradiction.

So we have found parts $V_t, V_{t'}$ containing~$Q$ and~$P_0$, respectively. Choose two such parts at minimum distance in~$T$. Note first that $t \neq t'$, because $P_0 \cup Q$ has size~$6n$ and we need at least one further vertex, for example~$a$, to connect these two paths. We now distinguish two cases. Suppose first that another node~$s$ of~$T$ lies between~$V_t$ and $V_{t'}$. By our choice of $t, t'$, there must be some $x^0_j, y_k \notin V_s$. But~$V_s$ separates~$V_t$ and~$V_{t'}$, so it must contain some vertex from~$P^0_{j,k}$. Being connected, $V_s$ is actually contained in this path. But then it cannot separate any other two vertices of~$P$ and~$Q$, which is a contradiction. Suppose now that~$t$ and~$t'$ are neighbours in~$T$. Pick any $x^0_p \in P_0 \setminus V_t$ and $y_q \in Q \setminus V_{t'}$. Since~$V_t \cap V_{t'}$ separates the two, it contains some vertex of~$P_{p,q}^0$, and thus at least one of~$V_t$, $V_{t'}$ contains at least half the vertices of~$P^0_{p,q}$. We may assume that this applies to~$Q$; the other case follows symmetrically. For every $1 \leq j \leq 2n$ consider $R_j :=\bigcup_{k = 1}^{4n} P^0_{j,k} \setminus Q$, the subdivision of a star with root~$x^0_j$. These are pairwise disjoint and disjoint from~$Q$, and since~$V_t$ contains at least $n/2 \geq 2$ vertices from~$R_p$, there is some~$m \in [2n]$ with $V_t \cap R_m = \emptyset$, by our assumption on the width. As $V_t \cap V_{t'}$ separates~$x_m^0$ from~$V_t$, we must have $Q \sub V_{t'}$, contradicting $t \neq t'$.

\end{section}

\begin{section}{Concluding Remarks and open problems} \label{remarks}
	Define the \emph{connected bramble number} $\cbn(G)$ of a graph~$G$ to be the maximum connected order of any bramble in~$G$. In Section~\ref{brambles} we observed that
	\begin{equation}	
		\cbn(G) - 1 \leq \wctw(G) \leq \ctw(G) \label{double ineq}
		\tag{$\dagger$}
	\end{equation}
holds for any graph~$G$. Diestel and M\"uller~\cite{MalteRhd} conjectured that $\cbn(G) - 1 = \ctw(G)$, but our example in Section~\ref{duality} shows that the second of the two inequalities in~(\ref{double ineq}) cannot be replaced by an equality. We do suspect, however, that the first inequality is in fact an equality:

\begin{question}
	Let~$k$ be a positive integer. A graph~$G$ has a \td\ in which every part is contained in a connected set of at most~$k$ vertices
	if and only if every bramble of~$G$ can be covered by a connected set of size at most~$k$.
\end{question}

%\begin{question}
%	Let~$k$ be a positive integer and~$G$ a graph. Then exactly one of the following holds:
%	\begin{enumerate}[(i)]
%		\item	$G$ has a \td\ in which every part is contained in a connected set of at most~$k$ vertices.
%		\item $G$ has a bramble that cannot be covered by any connected set of size at most~$k$.
%	\end{enumerate}
%\end{question}

It seems that neither the proof techniques of ordinary tree-width duality nor the ideas underlying our counterexample to connected tree-width duality are apt to solve this problem; hence we are confident that an inquiry into this problem is going to provide new ideas and insights.

The second problem concerns the second inequality of~(\ref{double ineq}). The proof of~\cite[Theorem~1.2]{MalteRhd} combined with the improved bound of Theorem~\ref{main result} shows that $\ctw(G) \leq 2(\cbn(G) -1)(\cbn(G) -2)$, unless~$G$ is a forest in which case~$\ctw(G) = \tw(G)$. This implies a locality principle for connected tree-width: if there is a \td\ in which every part, individually, can be wrapped in a connected set of size at most~$k$, then there is a \td\ with connected parts of size at most $2(k-1)(k-2)$. It would be interesting to get a better understanding of this dependency.

\begin{question}
	Is there a constant $\alpha > 0$ such that for every graph~$G$ 
	\[
	\ctw(G) \leq \alpha \, \wctw(G) ?
	\]
\end{question}

Our example in Section~\ref{duality} shows that this is not true for any $\alpha < 6/5$.

\end{section}

\bibliographystyle{plain}
\bibliography{collectivebib}

\end{document}